\documentclass[11pt]{article}

\usepackage[utf8]{inputenc}
\usepackage[T1]{fontenc}
\usepackage{lmodern}
\usepackage{amsmath,amssymb,amsthm,mathtools}
\usepackage[numbers,sort&compress]{natbib}
\usepackage{booktabs}
\usepackage{microtype}
\usepackage{geometry}
\usepackage{xcolor}

\usepackage{hyperref}

\providecommand{\doi}[1]{%
  \href{https://doi.org/#1}{\nolinkurl{doi:#1}}%
}

\usepackage[nameinlink,capitalize,noabbrev]{cleveref}

\hypersetup{
  colorlinks=true,
  linkcolor=blue!45!black,
  citecolor=blue!45!black,
  urlcolor=blue!55!black,
pdftitle={A proof of the maximum Laplacian energy conjecture for connected graphs via a sharp eigenvalue-sum bound},
  pdfauthor={Seyed Ahmad Mojallal}
}

\newtheorem{theorem}{Theorem}[section]
\newtheorem{lemma}[theorem]{Lemma}
\newtheorem{conjecture}[theorem]{Conjecture}
\newtheorem{proposition}[theorem]{Proposition}
\newtheorem{corollary}[theorem]{Corollary}
\theoremstyle{definition}

\theoremstyle{remark}
\newtheorem{remark}[theorem]{Remark}

\numberwithin{equation}{section}

\newcommand{\LE}{\operatorname{LE}}
\newcommand{\PA}{\operatorname{PA}}
\newcommand{\Lspec}{\operatorname{Lspec}}
\newcommand{\dd}{\bar d}

\begin{document}

\title{A proof of the maximum Laplacian energy conjecture for connected
graphs via a sharp eigenvalue-sum bound}

\author{Seyed Ahmad Mojallal}
\date{}

\maketitle

\begin{abstract}
Let \(S_k(G)\) denote the sum of the \(k\) largest Laplacian eigenvalues
of a connected graph \(G\) of order \(n\) and size \(m\).  
Write \(\PA_{n,\omega}\) for the graph obtained from an
\(\omega\)-vertex clique by attaching \(n-\omega\) pendant vertices to
one of its vertices, and set
\[
  M_{n,k}:=\binom{k+1}{2}+n-k-1,
\]
the number of edges of \(\PA_{n,k+1}\).  For \(n/2<k\le n-2\), we prove
the sharp bound
\[
  S_k(G)\le
  \frac{2k}{n}m+
  \frac{2(n-k)}{n}M_{n,k}-(n-k-1),
\]
with equality attained by \(\PA_{n,k+1}\).  This bound is complementary
to Brouwer's inequality and is strictly stronger when \(m<M_{n,k}\).
Combining our bound with Brouwer's inequality, we resolve and strengthen
a conjecture of Vinagre, Del-Vecchio, Justo, and Trevisan: for every
\(n\), the pineapple \(\PA_{n,\,1+\lfloor2n/3\rfloor}\) maximizes the
Laplacian energy among all connected graphs of order \(n\); moreover,
for \(n>4\), it is the unique maximizer.
\end{abstract}


\medskip

\section{Introduction}\label{sec:introduction}

Let \(G\) be a finite simple graph with \(n\) vertices, \(m\) edges, and
Laplacian matrix \(L(G)\).  If
\[
  \mu_1(G)\ge \mu_2(G)\ge\cdots\ge\mu_n(G)=0
\]
are its Laplacian eigenvalues, then the \emph{Laplacian energy} of \(G\) is
\begin{equation}\label{eq:LE-definition}
  \LE(G)=\sum_{i=1}^{n}\left|\mu_i(G)-\frac{2m}{n}\right|.
\end{equation}
Gutman and Zhou introduced this invariant as a Laplacian analogue of graph
energy \cite{GutmanZhou2006}.  It is also the trace norm of the centered
Laplacian \(L(G)-(2m/n)I\).  Thus \(\LE(G)\) records the total spectral
distribution about the average degree of the graph.  Determining its extremal behavior is
therefore a natural meeting point of spectral graph theory, extremal graph
theory, and matrix inequalities.

Maximizing Laplacian energy among connected graphs of fixed order is
delicate, since adding an edge changes both the Laplacian spectrum and
the average degree \(2m/n\).  Consequently, Laplacian energy is not
monotone under edge addition.  The conjectured maximizer is a pineapple
graph, which combines a complete subgraph with a collection of pendant
vertices attached to a single clique vertex.

For integers \(2\le\omega\le n\), let
\(\PA_{n,\omega}\) denote the \emph{pineapple graph} obtained from the complete graph 
\(K_\omega\) by attaching all \(n-\omega\) pendant vertices to one vertex
of the clique.  Vinagre, Del-Vecchio, Justo, and Trevisan
\cite{VinagreEtAl2013} proposed the following conjecture.

\begin{conjecture}
\label{conj:pineapple}
For every connected graph \(G\) of order \(n\),
\begin{equation}\label{eq:pineapple-conjecture}
  \LE(G)\le
  \LE\!\left(\PA_{n,\,1+\lfloor 2n/3\rfloor}\right).
\end{equation}
\end{conjecture}

Threshold graphs supplied both the origin and the strongest early
evidence for this conjecture.  Such graphs may be constructed from one
vertex by repeatedly adding either an isolated vertex or a vertex adjacent
to all earlier vertices.  Their Laplacian spectra equal their conjugate
degree sequences \cite{Merris1994,HammerKelmans1996}, which makes spectral
optimization unusually explicit.  Vinagre et al.\ proved the pineapple
optimal in several substantial threshold families and among all
pineapples \cite{VinagreEtAl2013}.  Helmberg and Trevisan subsequently
determined the maximum among all connected threshold graphs
\cite{HelmbergTrevisan2015}.  Das and Mojallal gave a further extremal
analysis, including the second-largest values in that class
\cite{DasMojallal2016}.  These results settle the extremal problem within the class of connected
threshold graphs, but do not show that an arbitrary connected graph is
dominated in Laplacian energy by a connected threshold graph.

The requirement that the comparison graph remain connected is crucial.
Indeed, when disconnected threshold graphs are allowed, the
Laplacian-energy maximizer is different and need not be connected
\cite{HelmbergTrevisan2015}.  Helmberg and Trevisan subsequently
introduced \emph{spectral threshold dominance}.
For \(1\le k\le n\), write
\[
  S_k(G)=\sum_{i=1}^{k}\mu_i(G).
\]
A graph \(G\) of order \(n\) and size \(m\) is called
\emph{spectrally threshold dominated} if, for every
\(k\in\{1,\ldots,n\}\), there exists a threshold graph \(T_k\) of
order \(n\) and size \(m\) such that
\[
  S_k(G)\le S_k(T_k).
\]
The comparison graph \(T_k\) may depend on \(k\).
Helmberg and Trevisan proved that every
spectrally threshold-dominated graph \(G\) admits a threshold graph of
the same order and size whose Laplacian energy is at least
\(\LE(G)\) \cite{HelmbergTrevisan2017}.  They further showed that
universal spectral threshold dominance is equivalent to Brouwer's
conjecture for Laplacian eigenvalue sums.  However, even when \(G\) is
connected, the resulting threshold graph need not be connected.
Consequently, spectral threshold dominance does not resolve the
connected extremal problem and, in particular, does not imply
\eqref{eq:pineapple-conjecture}.

Brouwer's inequality is
\begin{equation}\label{eq:Brouwer-intro}
  S_k(G)\le m+\binom{k+1}{2}
  \qquad (1\le k\le n).
\end{equation}
The inequality emerged from work on partial Laplacian spectral sums
\cite{HaemersEtAl2010,BrouwerHaemers2012} and remained open after many
partial results; see, for example, \cite{GanieEtAl2020}.  Recently,
Kothari and Tudose proved Brouwer's inequality using the
Grone--Merris--Bai majorization theorem
\cite{Bai2011,KothariTudose2026}.  Their theorem is the principal
external input used in our proof.

At first sight, \eqref{eq:Brouwer-intro} appears strong enough for
\eqref{eq:pineapple-conjecture}.  The obstruction is numerical rather than
formal.  Laplacian energy is determined by \(S_k-2mk/n\), where \(k\) is
the number of eigenvalues above \(2m/n\). 

When \(k>n/2\), the estimate obtained from Brouwer's inequality after
subtracting \(2km/n\) is decreasing in \(m\), and is therefore
insufficient below the pineapple edge count.  This suggests seeking a
complementary bound whose coefficient of \(m\) is \(2k/n\), so that the
dependence on \(m\) disappears in the corresponding Laplacian-energy
estimate.

\vspace{3mm}

The main contribution of this paper is precisely such a bound.  For
\(n/2<k\le n-2\), set
\begin{equation}\label{eq:Mnk-intro}
  M_{n,k}=\binom{k+1}{2}+n-k-1.
\end{equation}
Using \eqref{eq:Brouwer-intro}, we prove that every connected
graph of order \(n\) and size \(m\) satisfies
\begin{equation}\label{eq:CB-intro}
  S_k(G)\le
  \frac{2k}{n}m+
  \frac{2(n-k)}{n}M_{n,k}-(n-k-1).
\end{equation}
This inequality is sharp at \(\PA_{n,k+1}\).  If \(U_B\) and \(U_C\)
denote the right-hand sides of \eqref{eq:Brouwer-intro} and
\eqref{eq:CB-intro}, respectively, then
\begin{equation}\label{eq:bounds-difference-intro}
  U_B-U_C
  =\frac{2k-n}{n}\bigl(M_{n,k}-m\bigr).
\end{equation}
Thus the new estimate is stronger for \(m<M_{n,k}\), the two estimates
meet at \(m=M_{n,k}\), and Brouwer's estimate is stronger for
\(m>M_{n,k}\). 

The proof of \eqref{eq:CB-intro} is driven by the bottom, rather than the
top, of the Laplacian spectrum.  Write \(r=n-k-1\) and let \(B\) be the
sum of the \(r\) smallest positive Laplacian eigenvalues.  Since
\(S_k=2m-B\), the desired upper bound is a lower bound for \(B\).

Combining \eqref{eq:CB-intro} with Brouwer's inequality and optimizing
over \(k\) proves \eqref{eq:pineapple-conjecture}, with optimal pineapple
parameter \(1+\lfloor2n/3\rfloor\).  The maximum inequality uses the
theorem of Kothari and Tudose \cite{KothariTudose2026}.  The uniqueness
statement additionally uses the equality characterization for Brouwer's
inequality, proved independently in
\cite{CaiEtAl2026,CuiChen2026}.

The paper is organized as follows.  \Cref{sec:preliminaries} records the
Laplacian-energy identity and the pineapple spectrum, and states the main results. \Cref{sec:tail} gives all necessary results and the proof for the complementary eigenvalue-sum bound.  \Cref{sec:energy} converts that bound
into the extremal Laplacian-energy theorem and optimizes the comparison
pineapples.  \Cref{sec:equality} treats equality and uniqueness.

\section{Preliminaries and main results}\label{sec:preliminaries}

All graphs in this paper are finite and simple.  The size of \(G\) is
\(m=|E(G)|\), and its average degree is
\(\dd=2m/n\).  Laplacian eigenvalues are always listed in nonincreasing
order.

\begin{lemma}\label{lem:critical-index}
Let \(G\) have at least one edge, and put
\[
  \sigma^+=|\{i:\mu_i(G)>\dd\}|.
\]
Then \(1\le \sigma^+\le n-1\) and
\begin{equation}\label{eq:critical-index}
  \LE(G)=2 \Big(S_{\sigma^+}(G)-\frac{2m\, \sigma^+}{n}\Big).
\end{equation}
\end{lemma}

\begin{proof}
We have $$\sum_{i=1}^n (\mu_i(G)-\dd)=0.$$ This gives $$\sum_{\mu_i>\dd}(\mu_i-\dd)=-\sum_{\mu_i<\dd}(\mu_i-\dd).$$
Therefore, \[
  \frac{\LE(G)}2
  =\sum_{\mu_i>\dd}(\mu_i-\dd)
  =S_{\sigma^+}-\sigma^+ \dd.
\]

Since \(m\ge1\), we have
\[
  \mu_1(G)\ge\frac{2m}{n-1}>\frac{2m}{n}=\dd>0=\mu_n(G).
\]
Hence \(1\le\sigma^+\le n-1\).
\end{proof}

\vspace{3mm}

Taking \(t=\omega-1\), Das and Mojallal
\cite[Lemma~4.1]{DasMojallal2016} obtained
\[
  \LE(\PA_{n,\omega})
  =
  -\frac{2t^3}{n}
  +\left(2+\frac{2}{n}\right)t^2
  +\left(\frac{4}{n}-4\right)t
  +2n-2.
\]
Equivalently, under the change of variables
$x=n-\omega+1=n-t$, this identity becomes \cref{eq:fn-definition}. We write \(\Lspec(G)\) for the Laplacian spectrum of \(G\), with
superscripts in square brackets indicating multiplicities.
\begin{lemma}\label{lem:pineapple-data}
For \(2\le \omega\le n\), the pineapple \(\PA_{n,\omega}\) has $\binom{\omega}{2}+n-\omega$ edges, and
$$\Lspec(\PA_{n,\omega})=\{n,\ \omega^{[\omega-2]},\ 1^{[n-\omega]},\ 0\}.
$$ Moreover, for \(1\le x\le n-1\),
\begin{equation}\label{eq:fn-definition}
  \frac{\LE(\PA_{n,n-x+1})}{2}
  =f_n(x):=
  \frac{2x}{n}
  \left[\binom{n-x+1}{2}+x-1\right]-(x-1).
\end{equation}
\end{lemma}

\vspace{3mm}

We shall use the following theorem of Kothari and Tudose.
\begin{theorem}[Kothari--Tudose \cite{KothariTudose2026}]
\label{thm:Brouwer}
For every simple graph \(G\) of order \(n\) and size \(m\),
\begin{equation}\label{eq:Brouwer}
  S_k(G)\le m+\binom{k+1}{2}
  \qquad (1\le k\le n).
\end{equation}
\end{theorem}

\vspace{3mm}

Our first main result is an upper bound on $S_k(G)$ for $\frac{n}2< k\le n-2$.
\begin{theorem}[Complementary Brouwer bound]\label{thm:CB}
 Let \(G\) be a connected graph of order \(n\)
and size \(m\).  For every
integer \(k\) satisfying \(\frac{n}2<k\le n-2\), define
\[
  M_{n,k}=\binom{k+1}{2}+n-k-1.
\]
Then
\begin{equation}\label{eq:CB}
  S_k(G)\le
  \frac{2k}{n}m+
  \frac{2(n-k)}{n}M_{n,k}-(n-k-1).
\end{equation}

For every admissible $k$, equality is attained by \(\PA_{n,k+1}\).
\end{theorem}

\vspace{3mm}

We now apply \cref{thm:CB} to prove the connected pineapple conjecture, the second main result of the paper.

\begin{theorem}\label{thm:main}
Let \(G\) be a connected graph of order \(n\ge 2\). Then
\begin{equation}\label{eq:main}
  \LE(G)\le
  \LE(\PA_{n,\,1+\lfloor2n/3\rfloor}).
\end{equation}
\end{theorem}

\section{Proof of the complementary eigenvalue-sum bound}
\label{sec:tail}

This section proves \cref{thm:CB}.  We first record the precise
consequence of Bogdanowicz's theorem that is needed.

%\subsection{The minimum number of spanning trees}

For a connected graph \(G\), let \(t(G)\) denote its number of spanning trees.

\begin{proposition}\label{prop:Bogdanowicz}
Let \(p\ge4\), \(r\ge1\), and \(2\le h\le p-2\).  If \(G\) is connected with
\[
  n=p+r,
  \qquad
  m=\binom{p-1}{2}+h+r,
\]
then
\begin{equation}\label{eq:tau-lower-h}
  t(G)\ge h\,p^{h-1}(p-1)^{p-h-2}.
\end{equation}
\end{proposition}

\begin{proof}
Bogdanowicz proved that, among all connected graphs of fixed order \(n\)
and size \(m\), the minimum number of spanning trees is attained by a
canonical threshold graph \(L_{n,m}\) \cite{Bogdanowicz2009}.  For the
parameters above, \(L_{n,m}\) consists of a clique \(K_{p-1}\), one
additional vertex adjacent to exactly \(h\) clique vertices, and \(r\)
pendant vertices.  Applying Bogdanowicz's spanning-tree formula to this
graph gives
\[
  t(L_{n,m})
  =h\,p^{h-1}(p-1)^{p-h-2}.
\]
Consequently,
\[
  t(G)\ge t(L_{n,m})
  =h\,p^{h-1}(p-1)^{p-h-2},
\]
as required.
\end{proof}

\vspace{3mm}
%\subsection{The scalar ratio lemma}

The proof of the main result of this section requires the following lemma.

\begin{lemma}\label{lem:ratio}
Let \(r\ge1\) and \(p\ge r+3\), and put
\[
  n=p+r,\qquad
  s=\frac{2(r+1)}n,\qquad
  \alpha=\frac sr,\qquad
  D=\frac1\alpha.
\]
For \(1\le d<D\), define
\begin{equation}\label{eq:Ad}
  A_d=1+\frac{r}{p(p-1)}-\frac{2d}{np}
\end{equation}
and
\begin{equation}\label{eq:Rd}
  \mathcal R_d=
  \frac{p^2}{n(p-d-1)}
  (1-\alpha d)^r A_d^{p-1}
  \left(\frac{p}{p-1}\right)^{d-1}.
\end{equation}
Then \(\mathcal R_d<1\).
\end{lemma}

\begin{proof}
Write \(q=p-r-2\ge1\).  Since
\[
  rn-2(r+1)=r(2r+q)-2>0,
  \qquad
  2(r+1)(p-2)-rn=q(r+2)>0,
\]
we have \(1<D<p-2\).  Also, for \(d<D\),
\[
  A_d>
  1+\frac{r}{p(p-1)}-\frac{r}{p(r+1)}
  >1-\frac1p>0.
\]

Let \(\ell(d)=\log\mathcal R_d\).  At \(d=1\), we have
\[
  \ell(1)
  =
  \log\frac{p^2}{n(p-2)}
  +r\log(1-\alpha)
  +(p-1)\log A_1.
\]
Since \(0<\alpha<1\), the inequalities
\(\log x\le x-1\) and \(\log(1-\alpha)<-\alpha\) yield
\begin{align}
  \ell(1)
  &<
  \left(\frac{p^2}{n(p-2)}-1\right)
  -r\alpha+(p-1)(A_1-1)\notag\\
  &=
  \frac{p^2}{n(p-2)}-1
  -\frac{2(r+1)}{n}
  +\frac{r}{p}
  -\frac{2(p-1)}{np}\notag\\
  &=
  \frac{
    p\bigl[p^2-n(p-2)\bigr]
    -2(r+1)p(p-2)
    +rn(p-2)
    -2(p-1)(p-2)
  }{np(p-2)}\notag\\
  &=
  \frac{N(r,p)}{np(p-2)},
  \label{eq:ell-one}
\end{align}
where, after substituting \(n=p+r\) and simplifying,
\begin{equation}\label{eq:N-rp}
  N(r,p)
  =
  -2(r+1)p^2
  +(r^2+4r+10)p
  -2r^2-4.
\end{equation}

With \(p=r+q+2\),
\[
  -N(r,p)
  =r^3+9r^2+8r-8
   +(q-1)\!\left[2(r+1)(q+1)+3r^2+8r-2\right]>0,
\]
so \(\ell(1)<0\).

Differentiation gives
\[
  \ell'(d)=
  \frac1{p-d-1}
  -\frac{s}{1-\alpha d}
  -\frac{2(p-1)}{npA_d}
  +\log\frac{p}{p-1}.
\]
Since \(1\le d<D\) and \(\log(p/(p-1))<1/(p-1)\),
\[
  \ell'(d)<
  -\frac{s}{1-\alpha}
  +\frac1{p-D-1}
  +\frac1{p-1}.
\]
Set
\[
  H=r(2r+q)-2,\qquad
  Z=q(r+2)+2(r+1),\qquad
  Y=r+q+1.
\] The preceding inequalities show that \(H>0\), while \(Z,Y>0\) are
immediate.
A direct simplification yields
\[
  \frac{s}{1-\alpha}
  -\frac1{p-D-1}
  -\frac1{p-1}
  =
  \frac{
  r^2(2r+3)q^2
  +2(r^4+2r^3+2r^2+5r+4)q
  +8(r+1)^2}{HZY}>0.
\]
Thus \(\ell'(d)<0\) on \([1,D)\), and hence
\(\ell(d)\le\ell(1)<0\).
\end{proof}

%\subsection{The hard-deficit tail estimate}

\vspace{3mm}
The proof of the main eigenvalue-sum bound requires the following
estimate for the sum of the \(r\) smallest positive Laplacian eigenvalues.

\begin{theorem}\label{thm:tail}
Let \(p\ge r+3\) and \(r\ge1\) be integers, let \(n=p+r\), and put
\[
  s=\frac{2(r+1)}n.
\]
Suppose that \(G\) is connected, has order \(n\), and has size
\[
  m=\binom p2+r-d,
  \qquad
  d\in\mathbb Z,\qquad
  1\le d<\frac r s.
\]
If
$\mu_1\ge\mu_2\ge\cdots\ge\mu_{n-1}>\mu_n=0$ 
are the Laplacian eigenvalues of \(G\), then
\begin{equation}\label{eq:tail-stability}
  B_r(G):=\sum_{i=n-r}^{n-1}\mu_i>r-sd.
\end{equation}

\end{theorem}

\begin{proof}
Put
\[
  T=r-sd,\qquad
  h=p-d-1.
\]
Since \(p\ge r+3\),
\[
  2(r+1)(p-2)-rn=(p-r-2)(r+2)>0,
\]
so \(r/s<p-2\).  Since \(d\) is an integer with
\(1\le d<r/s\), we have \(2\le h\le p-2\).  Also
\[
  m=\binom{p-1}{2}+h+r.
\]
Hence \cref{prop:Bogdanowicz} gives
\begin{equation}\label{eq:tau-lower-d}
  t(G)\ge
  (p-d-1)p^{p-d-2}(p-1)^{d-1}
  =:t_0.
\end{equation}

Write \(B=B_r(G)\).  The Matrix--Tree theorem followed by AM--GM,
applied separately to the \(r\) smallest positive eigenvalues and the
remaining \(p-1\) eigenvalues, gives
\begin{equation}\label{eq:Phi}
  n\, t(G)
  \le
  \Phi_d(B):=
  \left(\frac Br\right)^r
  \left(\frac{2m-B}{p-1}\right)^{p-1}.
\end{equation}
Suppose, to the contrary, that \(B\le T\).  Since \(T<r\) and
\(m\ge n-1\),
\[
  \frac{d}{dB}\log\Phi_d(B)
  =\frac rB-\frac{p-1}{2m-B}>0
  \qquad(0<B\le T),
\]
because \(2mr\ge2(n-1)r>(n-1)B\).  Consequently,
\begin{equation}\label{eq:Phi-T}
  n\, t_0\le n\, t(G)\le\Phi_d(B)\le\Phi_d(T).
\end{equation}

Let \(\alpha=s/r\), \(D=1/\alpha\), and let \(A_d\) be as in
\cref{eq:Ad}.  A direct calculation gives
\[
  \frac Tr=1-\alpha d,
  \qquad
  \frac{2m-T}{p-1}=pA_d.
\]
Therefore
\begin{equation}\label{eq:ratio-identification}
  \frac{\Phi_d(T)}{n\, t_0}
  =
  \frac{p^2}{n(p-d-1)}
  (1-\alpha d)^r A_d^{p-1}
  \left(\frac p{p-1}\right)^{d-1}
  =\mathcal R_d.
\end{equation}
By \cref{lem:ratio}, the right-hand side is strictly smaller than one,
contradicting \cref{eq:Phi-T}.  Thus \(B>T=r-sd\).
\end{proof}

\subsection{Proof of the complementary bound}

\begin{proof}[Proof of \cref{thm:CB}]
Fix \(n/2<k\le n-2\), and put
\[
  p=k+1,\qquad
  r=n-k-1,\qquad
  s=\frac{2(n-k)}n=\frac{2(r+1)}n.
\]
The upper restriction on \(k\) gives \(r\ge1\), while \(k>n/2\) gives
\[
  p-r-2=2k-n\ge1,
  \qquad
  0<s<1;
\]
hence \(p\ge r+3\).  Set
\[
  C=\binom{k+1}{2},
  \qquad
  M=C+r=M_{n,k},
  \qquad
  d=M-m,
\]
and let
\[
  B=\sum_{i=k+1}^{n-1}\mu_i(G).
\]
Since \(B=2m-S_k(G)\), it is enough to show
\begin{equation}\label{eq:B-target}
  B\ge r-sd.
\end{equation}

If \(d\le0\), Brouwer's inequality gives
\[
  B\ge m-C=r-d\ge r-sd.
\]
If \(d\ge r/s\), connectedness gives \(B>0\ge r-sd\).
If \(1\le d<r/s\), \cref{thm:tail} gives \(B>r-sd\).
Thus \cref{eq:B-target} holds in every case.

Finally, \(d=M-m\), \(r=n-k-1\), and \(2-s=2k/n\) give
\[
  S_k(G)=2m-B
  \le
  \frac{2k}{n}m+
  \frac{2(n-k)}nM_{n,k}-(n-k-1).
\]
For \(\PA_{n,k+1}\), \(m=M_{n,k}\) and the bottom \(r\) positive
Laplacian eigenvalues are all \(1\), so equality holds.
\end{proof}

\section{Maximum Laplacian energy}
\label{sec:energy}
We first settle the small orders, allowing the remainder of the argument
to assume \(n>4\).  This also explains why uniqueness must be restricted
to that range.
\begin{proposition}[Small orders]\label{prop:small-orders}
Let \(G\) be a connected graph of order \(2\le n\le4\).  Then
\[
  \LE(G)\le
  \LE\!\left(\PA_{n,\,1+\lfloor2n/3\rfloor}\right)=2n-2.
\]
For \(n=2\) and \(n=3\), equality holds only for \(K_2\) and \(K_3\),
respectively.  For \(n=4\), equality holds precisely for
\[
  G\cong\PA_{4,3},\qquad G\cong K_4-e,\qquad\text{or}\qquad G\cong K_4.
\]
\end{proposition}

\begin{proof}
For \(n=2\), the only connected graph is
\(\PA_{2,2}=K_2\), whose Laplacian energy is \(2\).  For \(n=3\), the
connected graphs are \(P_3\) and \(K_3=\PA_{3,3}\), and
\[
  \LE(P_3)=\frac{10}{3}<4=\LE(K_3).
\]
For \(n=4\), the six connected graphs, up to isomorphism, have
Laplacian energies
\[
\begin{array}{c|cccccc}
G
 & P_4 & C_4 & K_{1,3} & \PA_{4,3} & K_4-e & K_4\\
\hline
\LE(G)
 & 2+2\sqrt2 & 4 & 5 & 6 & 6 & 6.
\end{array}
\]
This proves the result.
\end{proof}

\begin{proof}[Proof of \cref{thm:main}] By \cref{prop:small-orders}, it remains to consider \(n>4\).
 Let \(G\) be connected of order \(n>4\),
and let \(\sigma^+\) be the index from
\cref{lem:critical-index}.  We split the proof according to \(\sigma^+\).

%\subsection{Small critical index}
\vspace{3mm}

Suppose first that \(1\le \sigma^+\le n/2\).  It is well-known that every Laplacian eigenvalue of
a simple graph is at most \(n\). 
Combining this with \cref{eq:Brouwer,eq:critical-index} gives
\begin{equation}\label{eq:low-envelope}
  \frac{\LE(G)}2
  \le
  \min\left\{
    m+\binom{\sigma^+ +1}{2},\, \sigma^+ n
  \right\}
  -\frac{2m \sigma^+}{n}.
\end{equation}
The first branch $$m+\binom{\sigma^+ +1}{2}-\frac{2m \sigma^+}{n}$$ is nondecreasing in \(m\), while the second one $\sigma^+ n-\frac{2m \sigma^+}{n}$ is
decreasing.  Their intersection occurs at
\[
  m_0=\sigma^+ n-\binom{\sigma^+ +1}{2}.
\]
It follows that
\begin{equation}\label{eq:gn}
  \frac{\LE(G)}2
  \le
  g_n(\sigma^+):=
  \sigma^+(n-2 \sigma^+)+\frac{(\sigma^+)^2(\sigma^+ +1)}n.
\end{equation}
The same formula remains valid at \(\sigma^+=n/2\), where the first branch
in \cref{eq:low-envelope} is constant.

%\subsection{Large critical index}
\vspace{3mm}

At the endpoint \(\sigma^+=n-1\), \(S_{\sigma^+}=2m\), and therefore
\begin{equation}\label{eq:endpoint}
  \frac{\LE(G)}2=\frac{2m}{n}\le n-1=f_n(1),
\end{equation} where \(f_n\) is defined in \cref{eq:fn-definition}.

\vspace{3mm}

Now suppose \(n/2< \sigma^+\le n-2\).  Apply \cref{thm:CB} with
\(k=\sigma^+\).  If
\[
  r=n-\sigma^+ -1,
  \qquad
  M=M_{n,\sigma^+},
  \qquad
  s=\frac{2(n-\sigma^+)}n,
\]
then \cref{eq:critical-index,eq:CB} give
\begin{equation}\label{eq:high-energy}
  \frac{\LE(G)}2
  \le sM-r
  =f_n(n- \sigma^+).
\end{equation}

%\subsection{Discrete optimization}

For every integer \(1\le x\le\lfloor n/2\rfloor\), direct calculation from
\cref{eq:fn-definition,eq:gn} gives
\begin{equation}\label{eq:f-minus-g}
  f_n(x)-g_n(x)=1-\frac{2x}{n}\ge0.
\end{equation}
The small-index estimate \cref{eq:gn}, the endpoint estimate
\cref{eq:endpoint}, and the large-index estimate
\cref{eq:high-energy} therefore imply
\begin{equation}\label{eq:global-f}
  \frac{\LE(G)}2
  \le
  \max_{1\le x\le\lfloor n/2\rfloor}f_n(x).
\end{equation}

The forward difference of \(f_n\) is
\begin{equation}\label{eq:Qn}
  n\bigl(f_n(x+1)-f_n(x)\bigr)
  =Q_n(x):=
  3x^2+(5-4n)x+n^2-2n.
\end{equation}
On \(1\le x\le n/2\), the polynomial \(Q_n\) is nonincreasing because
\[
  Q_n'(x)=6x+5-4n\le5-n\le0.
\]
Put \(x_0=\lceil n/3\rceil\).  The signs around \(x_0\) are explicit:
\begin{equation}\label{eq:sign-table}
\begin{array}{c|cc}
  n & Q_n(x_0-1) & Q_n(x_0)\\
  \midrule
  3t   & 5t-2 & -t\\
  3t+1 & t-1  & 3-5t\\
  3t+2 & 3t   & -3t
\end{array}
\end{equation}
For \(n>4\), the left entry is positive and the right entry is negative
in every row.  Thus \(f_n\) has its unique maximum at
\[
  x_0=\left\lceil\frac n3\right\rceil.
\]
Finally,
\[
  n-x_0+1
  =1+\left\lfloor\frac{2n}{3}\right\rfloor.
\]
Together with \cref{eq:global-f,eq:fn-definition}, this proves
\cref{thm:main}.
\end{proof}

\vspace{3mm}

\begin{corollary}\label{cor:exact-value}
The maximum Laplacian energy of
a connected graph of order \(n>4\) is
\begin{equation}\label{eq:piecewise-maximum}
\LE(\PA_{n,\,1+\lfloor2n/3\rfloor})=
\begin{cases}
\displaystyle
\frac{8}{27}n^2+\frac{2}{9}n+\frac23,
&n\equiv0\pmod3,\\[6pt]
\displaystyle
\frac{8}{27}n^2+\frac{2}{9}n+\frac23-\frac{32}{27n},
&n\equiv1\pmod3,\\[6pt]
\displaystyle
\frac{8}{27}n^2+\frac{2}{9}n+\frac89-\frac{28}{27n},
&n\equiv2\pmod3.
\end{cases}
\end{equation}
\end{corollary}

\begin{proof}
Substitute \(x=\lceil n/3\rceil\) into
\cref{eq:fn-definition} and separate the three residue classes.
\end{proof}

\section{Equality and uniqueness}\label{sec:equality}

The proof of the maximum inequality requires only Brouwer's inequality.
For uniqueness, we additionally use the following equality
characterization, proved independently by Cai et al.\ and by Cui and
Chen \cite{CaiEtAl2026,CuiChen2026}: for \(1\le k\le n-1\),
\begin{equation}\label{eq:Brouwer-equality}
  S_k(G)=m+\binom{k+1}{2}
  \quad\Longleftrightarrow\quad
  G\text{ is a threshold graph with clique number }k+1.
\end{equation}

\begin{theorem}\label{thm:uniqueness}
Let \(G\) be a connected graph of
order \(n>4\). Then
\[
  \LE(G)=\LE(\PA_{n,\,1+\lfloor2n/3\rfloor})
  \quad\Longleftrightarrow\quad
  G\cong \PA_{n,\,1+\lfloor2n/3\rfloor}.
\]
\end{theorem}

\begin{proof}
Suppose \(\LE(G)=\LE(\PA_{n,\,1+\lfloor2n/3\rfloor})\).  If \(\sigma^+\le n/2\), then
\(g_n(\sigma^+)<f_n(x_0)\).  Indeed, this follows from the uniqueness
of the maximizer \(x_0\) when \(\sigma^+\ne x_0\), while for
\(\sigma^+=x_0\) it follows from \cref{eq:f-minus-g} and
\(x_0<n/2\).  Similarly, \cref{eq:endpoint} and the uniqueness of
\(x_0\) exclude \(\sigma^+=n-1\).

Hence
\[
  \frac n2< \sigma^+ \le n-2,
  \qquad
  n-\sigma^+=\left\lceil\frac n3\right\rceil.
\]

Use the notation from the proof of \cref{thm:CB}, with \(k=\sigma^+\).
Every case in that proof is strict unless \(d=0\).  Indeed, if \(d<0\),
then \(r-d>r-sd\); if \(1\le d<r/s\), strictness follows from
\cref{thm:tail}; and if \(d\ge r/s\), then \(B>0\ge r-sd\).
Thus equality forces
\[
  m=M_{n,\sigma^+},
  \qquad
  S_{\sigma^+}(G)=m+\binom{\sigma^+ +1}{2}.
\]
By \cref{eq:Brouwer-equality}, \(G\) is a threshold graph with clique number
\(\sigma^+ +1\).

 \vspace{3mm}
 
Take a split partition \(V(G)=K\cup I\) with
\(|K|=\sigma^+ +1\).  Then
\[
  |I|=n-\sigma^+ -1,
  \qquad
  m=\binom{\sigma^+ +1}{2}+|I|.
\]
Connectedness forces every vertex of \(I\) to meet \(K\), so the edge
count forces every such vertex to have exactly one neighbor.  These
neighborhoods are nested in a threshold graph and hence coincide.
Therefore all vertices of \(I\) are pendant at one clique vertex, and
\[
  G\cong\PA_{n,{\sigma^+}+1}=\PA_{n,\,1+\lfloor2n/3\rfloor}.
\]
The converse is immediate, since Laplacian energy is invariant under
graph isomorphism.
\end{proof}

\begin{remark}\label{rem:n-four}
The restriction \(n>4\) is necessary for uniqueness, since
\(\PA_{4,3}\), \(K_4-e\), and \(K_4\) all have Laplacian energy \(6\).
\end{remark}

\section*{Acknowledgements}
Seyed Ahmad Mojallal is partially supported by the ERC Synergy grant (European Union, ERC, KARST, project number 101071836). 

\section*{Declaration of AI use}

The author acknowledges the use of ChatGPT (GPT-5.5, OpenAI; accessed May 2026) solely for preliminary brainstorming and the exploration of possible proof strategies. All mathematical statements and proofs were reviewed and verified by the author, who takes full responsibility for the accuracy
and integrity of the article.

\bibliographystyle{unsrtnat}
\bibliography{Ref1}

\bigskip
\begin{flushleft}
\small
Seyed Ahmad Mojallal\\
Department of Mathematics, Simon Fraser University, Burnaby, BC, Canada\\
Email: \texttt{seyed\_ahmad\_mojallal@sfu.ca},
\texttt{ahmad\_mojalal@yahoo.com}
\end{flushleft}

\end{document}